\documentclass[a4paper]{amsart}
\usepackage{graphicx}
\usepackage{amssymb}
\usepackage{amsmath}
\usepackage{amsthm,amsfonts,bbm}
\usepackage{amscd}
\usepackage[all,2cell]{xy}

\UseAllTwocells \SilentMatrices

\newtheorem{thm}{Theorem}[section]

\newtheorem{cor}[thm]{Corollary}
\newtheorem{lem}[thm]{Lemma}

\newtheorem{prop-def}[thm]{Proposition-Definition}
\theoremstyle{definition}
\newtheorem{Def}[thm]{Definition}
\theoremstyle{remark}

\numberwithin{equation}{section}
\numberwithin{figure}{section}

\newcommand{\A}{\mathcal{A}}

\def\la{\lambda}

\def\A{\mathcal{A}}

\def\diag{{\rm diag}}

\def\x{{\mathbf x}}
\def\y{{\mathbf y}}

\def\C{\mathbb{C}}

\def\V{\mathcal{V}}

\def\I{\mathcal{I}}
\def\B{\mathcal{B}}

\def\diag{{\rm diag}}
\def \i{\mathbf{i}}
\def \A{\mathcal{A}}

\def \Dg{\mathfrak{D}}
\def \D{\mathcal{D}}
\def \Spec{\mbox{\rm Spec}}

\def \PV{\mathbb{V}}
\def\sp{\mbox{\rm supp}}

\begin{document}
\title[Dimension of Nonnegative Tensors]
{The Dimension of Eigenvariety of Nonnegative Tensors Associated with Spectral Radius}

\author[Y.-Z. Fan]{Yi-Zheng Fan$^*$}
\address{School of Mathematical Sciences, Anhui University, Hefei 230601, P. R. China}
\email{fanyz@ahu.edu.cn}
\thanks{$^*$The corresponding author.
The first and the second authors were supported by National Natural Science Foundation of China \#11371028.
The third author was supported by National Natural Science Foundation of China \#11401001.}

\author[T. Huang]{Tao Huang}
\address{School of Mathematical Sciences, Anhui University, Hefei 230601, P. R. China}
\email{huangtao@ahu.edu.cn}

\author[Y.-H. Bao]{Yan-Hong Bao}
\address{School of Mathematical Sciences, Anhui University, Hefei 230601, P. R. China}
\email{baoyh@ahu.edu.cn}
\date{\today}





\begin{abstract}
For a nonnegative weakly irreducible tensor,
  its spectral radius is an eigenvalue corresponding to a unique positive eigenvector up to a scalar called the Perron vector.
But including the Perron vector, it may have more than one eigenvector corresponding to the spectral radius.
The projective eigenvariety associated with the spectral radius is the set of the eigenvectors corresponding to the spectral radius considered in the complex projective space.
In this paper we prove that the dimension of the above projective eigenvariety is zero, i.e. there are finite many eigenvectors associated with the spectral radius up to a scalar.
For a general nonnegative tensor, we characterize the nonnegative combinatorially symmetric tensor for which the
 dimension of projective eigenvariety associated with spectral radius is greater than zero.
 Finally we apply those results to the adjacency tensors of uniform hypergraphs.
   \end{abstract}

\maketitle


\section{Introduction}

A \emph{tensor} (also called \emph{hypermatrix}) $\A=(a_{i_{1} i_2 \ldots i_{m}})$ of order $m$ and dimension $n$ over a field $k$ refers to a
 multi-array of entries $a_{i_{1}i_2\ldots i_{m}}\in k$
 for all $i_{j}\in [n]:=\{1,2,\ldots,n\}$ and $j\in [m]$.
 In this paper, we mainly consider complex tensors, i.e. $k=\C$.
For a given a vector $\x=(x_1, \cdots, x_n)\in \C^n$, $\A\x^{m-1}\in \C^n$, and is defined as
\begin{align*}
(\A\x^{m-1})_i=\sum_{i_2,\cdots, i_m\in [n]}a_{ii_2\cdots i_m}x_{i_2}\cdots x_{i_m}, i\in [n].
\end{align*}
Let $\I=(i_{i_1i_2\cdots i_m})$ be the \emph{identity tensor} of order $m$ and dimension $n$, that is,
$i_{i_1i_2\cdots i_m}=1$ for $i_1=i_2=\cdots=i_m$ and $i_{i_1i_2\cdots i_m}=0$ otherwise.
Recall an \emph{eigenvalue} $\la\in \C$ of $\A$ means that the polynomial system $(\la \I-\A)\x^{m-1}=0$,
  or equivalently the \emph{eigenequation} $\A\x^{m-1}=\la \x^{[m-1]}$, has a nontrivial solution $\x \in \C^n$ which is called an \emph{eigenvector} of $\A$ associated with $\la$, where $\x^{[m-1]}:=(x_1^{m-1}, \cdots, x_n^{m-1})$; see \cite{Lim,Qi1}.

The \emph{determinant} of $\A$, denoted by $\det \A$, is defined as the resultant of the polynomials $\A \x^{m-1}$,
and the \emph{characteristic polynomial} $\varphi_\A(\la)$ of $\A$ is defined as $\det(\la \I-\A)$ \cite{CPZ2,Qi1}.
It is known that $\la$ is an eigenvalue of $\A$ if and only if it is a root of $\varphi_\A(\la)$.
The \emph{algebraic multiplicity} of $\la$ as an eigenvalue of $\A$  is defined as the multiplicity of $\la$ as a root of $\varphi_\A(\la)$.
The \emph{spectrum} of $\A$, denoted by $\Spec(\A)$, is the multi-set of the roots of $\varphi_\A(\la)$.
The largest modulus of the elements in $\Spec(\A)$ is called the \emph{spectral radius} of $\A$, denoted by $\rho(\A)$.

Let $\la$ be an eigenvalue of $\A$ and $\V_\la=\V_\la(\A)$ the set of all eigenvectors of $\A$ associated with $\la$ together with zero, i.e.
$$\V_\la(\A)=\{\x \in \mathbb{C}^n: \A\x^{m-1}=\la \x^{[m-1]}\}.$$
Observe that the system of equations $(\la \I-\A)\x^{m-1}=0$ is not linear yet for $m\ge 3$,
and therefore $\V_\la$ is not a linear subspace of $\mathbb{C}^n$ in general.
In fact, $\V_\la$ forms an affine variety in $\C^n$ \cite{Ha}, which is called the \emph{eigenvariety}
of $\A$ associated with $\la$ \cite{HuYe}.
Let $\mathbb{P}^{n-1}$ be the standard complex projective spaces of dimension $n-1$.
Since each polynomial in the system $(\la \I-\A)\x^{m-1}=0$ is homogenous of degree $m-1$,
we consider the projective variety
$$\PV_\la=\PV_\la(\A)=\{\x \in \mathbb{P}^{n-1}: \A\x^{m-1}=\la \x^{[m-1]}\}.$$
which is called the \emph{projective eigenvariety} of $\A$ associated with $\la$.
%

By Perron-Frobenius theorem, it is known that if $\A$ is a nonnegative irreducible matrix,
then $\rho(\A)$ is a simple eigenvalue of $\A$ associated with a positive eigenvector called {\it Perron vector} up to a scalar.
So $\PV_{\rho(\A)}$ contains only one element.
But, in general for $\A$ being a nonnegative irreducible or weakly irreducible tensor,
including the Perron vector, $\A$ may have more than one eigenvector associated with the eigenvalue $\rho(\A)$ up to a scalar,
i.e. $\PV_{\rho(\A)}$ may contains more than one element \cite{FHB}.
So, \emph{it is a natural problem to characterize the dimension of $\PV_{\rho(\A)}$}.

Hu and Ye \cite{HuYe} define the geometric multiplicity of an eigenvalue $\la$ of $\A$ to be the dimension of $\V_\la(\A)$,
 and try to establish a relationship between the algebraic multiplicity and geometric multiplicity of an eigenvalue of $\A$.
For a nonnegative irreducible tensor $\A$, Li et.al \cite{LiYY} give a new definition of geometric multiplicity,
and their result implies that the dimension of $\PV_{\rho(\A)}$ is zero, i.e. there are a finite number of eigenvectors of $\A$  corresponding to $\rho(\A)$ up to a scalar.
Recently we show that for a nonnegative combinatorial symmetric weakly irreducible tensor $\A$,
the dimension of $\PV_{\rho(\A)}$ is zero \cite{FBH}.
Furthermore, the cardinality of $\PV_{\rho(\A)}$ can be get exactly by solving the Smith normal form of the incidence matrix of $\A$.
So, the next work is to consider the nonnegative weakly irreducible tensors or general nonnegative tensors.

%

In this paper by basic knowledge of algebraic geometry,
we proved that $\PV_{\rho(\A)}$ has dimension zero if $\A$ is nonnegative weakly irreducible.
We also give some upper bounds for the cardinality of $\PV_{\rho(\A)}$ for some special classes of nonnegative tensors $\A$,
and characterize the nonnegative combinatorially symmetric tensors $\A$ for which the
 dimension of $\PV_{\rho(\A)}$ is greater than zero.
Finally we apply those results to the adjacency tensors of uniform hypergraphs.

\section{Quasiprojective variety}
The basic knowledge of algebraic geometry in this section is adopted from the monographs \cite{Ha} and \cite{Shaf}.
\subsection{Affine space}
Let $k$ be an algebraically closed field.
Let $\mathbb{A}_k^n$ or simply $\mathbb{A}^n$ be an affine space of dimension $n$, the set of all $n$-tuples of elements of $k$.
Let $k[x]:=k[x_1,\ldots,x_n]$ be the polynomial ring in $n$ variables over $k$.
If $T$ is any subset of $k[x]$, define
$$V(T)=\{P \in \mathbb{A}^n: f(P)=0 ~\mbox{for all}~ f \in T\}.$$
A subset $X$ of $\mathbb{A}^n$ is an a \emph{closed subset} if there exists a subset $T \subseteq k[x]$ such that
$X=V(T)$.
The \emph{Zariski topology} on  $\mathbb{A}^n$ is the topology whose closed sets are exactly the closed sets in $\mathbb{A}^n$.
Let $X$ be a subset of $\mathbb{A}^n$.
The \emph{ideal of $X$} is defined to be
$$I(X)=\{f \in k[x]: f(P)=0 ~\mbox{for all}~ P \in X\}.$$
It is known that
\begin{enumerate}
\item $V(I(X))=\overline{X}$, where $\overline{X}$ is the closure of $X$ in the Zariski topology on $\mathbb{A}^n$;

\item (Hilbert Nullstellensatz) $I(V(\mathfrak{a}))=\sqrt{\mathfrak{a}}$ if $\mathfrak{a}$ is an ideal of $k[x]$.
\end{enumerate}

Let $X$ be a closed subset of $\mathbb{A}^n$.
A function $f: X \to k$ is called \emph{regular} if there
exists a polynomial $F \in k[x]$ such that $F(x)=f(x)$ for all $x \in X$.
The set of regular functions on $X$ is exactly the ring $k[X]:=k[x]/I(X)$, called the \emph{affine coordinate ring} of $X$.
A map $f: X \to Y$ between closed subsets of $\mathbb{A}^n$ and $\mathbb{A}^m$ respectively is \emph{regular} if there exist $m$ regular functions
$f_1,\ldots, f_m$ such that $f(x)=(f_1(x),\ldots,f_m(x))$ for all $x \in X$.


%
%

\subsection{Projective space}
The projective space $\mathbb{P}_k^n$ or simply $\mathbb{P}^n$ is the set of equivalent classes
$(\mathbb{A}_k^{n+1}\backslash \{(0,\ldots,0)\}) / \sim$, where $(a_0,\ldots,a_n) \sim (b_0,\ldots,b_n)$ if and only if there exists a nonzero $\la \in k$ such that
$b_i =\la a_i$ for  $i=0,\ldots,n$.
The class of $a=(a_0,\ldots,a_n)$ from $\mathbb{A}_k^{n+1}\backslash \{(0,\ldots,0)\}$ in $\mathbb{P}^n$ is denoted by
$[a_0:\cdots:a_n]$, called the \emph{homogeneous coordinates} of $a$.

Let $k[x]:=k[x_0,\ldots,x_n]$.
If $T$ is a subset of homogeneous polynomials in $k[x]$, define
$$V(T)=\{P \in \mathbb{P}^n: f(P)=0 \hbox{~for all~} f \in T\}.$$
A subset $X$ of $\mathbb{P}^n$ is a \emph{closed subset} if there exists a set $T$ of homogeneous polynomials of $k[x]$ such that
$X=V(T)$.
The \emph{Zariski topology} on  $\mathbb{P}^n$ is the topology whose closed sets are exactly the closed sets in $\mathbb{P}^n$.
An ideal $\mathfrak{a}$ of $k[x]$ is called a \emph{homogeneous ideal} if it is generated by a set of homogeneous polynomials.
For a subset $X$ of $\mathbb{P}^n$, the ideal of $X$ is defined as
$$I(X)=\{f \in k[x]: f(P)=0 ~\mbox{for all}~ P \in X\}.$$
It is known that $I(X)$ is a homogeneous ideal.
We still have analogous results to (1) and (2) in Section 2.1.

For each $i=0,\ldots,n$, set $\mathbb{A}_i^n:=\mathbb{P}^n \backslash V(x_i)$, called an \emph{affine piece} of $\mathbb{P}^n$.
Then $\mathbb{A}_i^n$ is obviously open, and the map (using $\mathbb{A}_0^n$ for example)
\begin{equation} \label{A0A} \phi_0: \mathbb{A}_0^n \to \mathbb{A}^n, [x_0:\cdots:x_n] \mapsto (x_1/x_0,\ldots,x_n/x_1),\end{equation}
is a homomorphism for the Zariski topology.
If $X$ is a closed set of $\mathbb{P}^n$ given by a system of homogeneous equations $F_1=\cdots=F_m=0$ and $\hbox{deg} F_j =n_j$, then
$\phi_0(X)$ is given by the system
$$ S_0^{-n_j} F_j=F_j(1,T_1,\ldots,T_n)=0, \hbox{~for~} j=1,\ldots,m,$$
where $T_i=S_i/S_0$ for $i=1,\ldots,n$.
If $U$ is a closed set of $\mathbb{A}^n$, then $U$ defines a closed projective set $\overline{U}$ called the \emph{projective completion of $U$}.
If $F(T_1,\ldots,T_n)$ is a polynomial in the ideal $I(U)$ with $\hbox{deg}F=k$, then the equations of $\overline{U}$ are of the form $S_0^kF(S_1/S_0,\ldots,S_n/S_0)$.
So
$ \phi_0^{-1}(U)=\overline{U} \cap \mathbb{A}_0^n$, a closed set in $\mathbb{A}_0^n$.
One can similarly define the map $\phi_i$ from $\mathbb{A}_i^n$ to $\mathbb{A}^n$ by replacing $x_0$ in (\ref{A0A}) by $x_i$ for $i=1,\ldots,n$.

\begin{Def}\label{quasi}
A \emph{quasiprojective variety} is an open subset of a closed projective set.
\end{Def}

By the above definition, the class of quasiprojective varieties includes all projective closed sets and all affine closed sets.
A closed subset of a quasiprojective variety is its intersection with a close set of projective space.
Open set and neighbourhood of a point are defined similarly.

Let $X \subseteq \mathbb{P}^n$, $x \in X$ and $f=P/Q$ be a homogeneous function of degree $0$ with $Q(x) \ne 0$, then
$f$ defines a function on a neighbourhood of $x$ in $X$ with values in $k$.
We say $f$ is \emph{regular} at $x$.
A function on $X$ that is regular at all points of $X$ is called a \emph{regular} function.
The set of regular functions on $X$ forms a ring, denoted by $k[X]$.
If $X$ is an affine closed set, then the definition of regular function here is the same as in Section 2.1.

\begin{Def}
Let $f: X \to Y$ be a map between quasiprojective varieties $X \subseteq \mathbb{P}^n$ and $Y  \subseteq \mathbb{P}^m$.
The map $f$ is \emph{regular} if for every $x \in X$, there exists some affine piece $\mathbb{A}_i^m$ containing $f(x)$ and a neighbourhood $U$ of $x$ such that
$f(U) \subseteq \mathbb{A}_i^m$ and the map $f|_U: U \to  \mathbb{A}_i^m$ is regular, i.e. given by $m$ regular functions from $k[U]$.
\end{Def}

A regular map between two quasiprojective varieties is called an \emph{isomorphism} if it has an inverse regular map.
It is known that the closed subsets map to closed subsets under isomorphism; see \cite[Section 4.2]{Shaf}.
An \emph{affine variety} (respectively, \emph{projective variety}) is a quasiprojective variety isomorphic to a closed subset of an affine space
(respectively, a closed subset of a projective space).

\subsection{Noetherian topological space}
A topological space $X$ is called \emph{noetherian} if it satisfies the descending chain condition for closed subsets:
for any sequence $Y_1 \supseteq Y_2 \supseteq \cdots$ of closed subsets, there is an integer $r$ such that $Y_r=Y_{r+1}=\cdots$.

A nonempty subset $Y$ of a topological space $X$ is called \emph{irreducible} if it cannot be expressed as the union $Y=Y_1 \cup Y_2$ of two
proper subsets, each one of which is closed in $Y$.

\begin{lem}\label{decom} \cite[Proposition 1.5]{Ha}
In a noetherian topological space $X$, every nonempty closed subset $Y$ can be expressed as a finite union $Y=Y_1 \cup \cdots Y_r$ of irreducible closed subsets $Y_i$.
If we require that $Y_i \nsupseteq Y_j$ if $i \ne j$, then the $Y_i$'s are unique determined. They are called the irreducible components of $Y$.
\end{lem}

The \emph{dimension} of a topological space $X$, denoted by $\dim X$, is defined to be the supremum of all integers $n$ such that there
exists a chain $Z_0 \subset Z_1 \subset \cdots \subset Z_n$ of distinct irreducible closed subsets of $X$.

As quasiprojective varieties (including projective closed sets and affine closed sets) are noetherian topological space with Zariski topology,
we have the decomposing result and the notion of dimension for a quasiprojective variety.

\subsection{Image of projective variety}
\begin{thm}\label{close} \cite[Theorem 1.10]{Shaf}
The image of a projective variety under a regular map is closed.
\end{thm}

\begin{cor}\label{one}\cite[Corollary 1.2]{Shaf}
A regular map $f: X \to Y$ from an irreducible projective variety $X$ to an affine variety $Y$ maps $X$ to a point.
\end{cor}

\begin{cor}\label{finite}
Let $f: X \to Y$ be a regular map from a projective variety $X$ to an affine variety $Y$.
Then $f(X)$ is finite.
\end{cor}

\begin{proof}
By Lemma \ref{decom}, we can write $X=X_1 \cup \cdots X_r$ of irreducible closed subsets $X_i$.
As $X$ is a projective variety, there is an isomorphism $\phi: X \to Y \subset \mathbb{P}^m$ for some $m$, where $Y$ is a closed set of $\mathbb{P}^m$;
and $\phi(X_i)$ is also a closed subset of $\mathbb{P}^m$, which is isomorphic to $X_i$.
So, each $X_i$ is an irreducible projective variety, and $f(X_i)$ is a point by Corollary \ref{one}, which implies that
$f(X)=\cup_{i=1}^r f(X_i)$ is finite.
\end{proof}

\section{Main result}

Let $\A=(a_{i_{1}i_2\ldots i_{m}})$ be a tensor of order $m$ and dimension $n$.
If all entries $a_{i_1i_2\cdots i_m}$ of $\A$ are invariant under any permutation of its indices, then $\A$ is called a \emph{symmetric tensor}.
The {\it support} of $\A$, also called the {\it zero-nonzero pattern} of $\A$ in \cite{Shao},
denoted by $\sp(\A)=(s_{i_1i_2\ldots i_m})$, is defined as a tensor with same order and dimension as $\A$, such that $s_{i_1\ldots i_m}=1$ if $a_{i_1\ldots i_m} \ne 0$,
and $s_{i_1\ldots i_m}=0$ otherwise.
The tensor $\A$ is called {\it combinatorial symmetric} if its support $\sp(\A)$ is symmetric.
The tensor $\A$ is called \emph{reducible} if there exists a nonempty proper index subset $I \subset [n]$ such that
$a_{i_{1}i_2\ldots i_{m}}=0$ for any $i_1 \in I$ and any $i_2,\ldots,i_m \notin I$;
if $\A$ is not reducible, then it is called \emph{irreducible} \cite{CPZ}.
We also can associate $\A$ with a directed graph $G(\A)$ on vertex set $[n]$ such that $(i,j)$ is an arc of $G(\A)$ if
and only if there exists a nonzero entry $ a_{ii_2\ldots i_{m}}$ such that $j \in \{ i_2,\ldots, i_{m}\}$.
Then $\A$ is called \emph{weakly irreducible} if $G(\A)$ is strongly connected; otherwise it is called \emph{weakly reducible} \cite{FGH}.
It is known that if $\A$ is irreducible, then it is weakly irreducible; but the converse is not true.

The following is the Perron-Frobenius theorem for nonnegative tensors,
where an eigenvalue is called \emph{$H^+$-eigenvalue}
(respectively, \emph{$H^{++}$-eigenvalue}) if it is associated with a nonnegative (respectively, positive) eigenvector.

\begin{thm}[Perron-Frobenius theorem for nonnegative tensors]\label{PF1}~~
\begin{enumerate}
\item{\em(Yang and Yang \cite{YY1})}  If $\A$ is a nonnegative tensor of order $m$ and dimension $n$, then $\rho(\A)$ is an $H^+$-eigenvalue of $\A$.

\item{\em(Friedland, Gaubert and Han \cite{FGH})} If furthermore $\A$ is weakly irreducible, then $\rho(\A)$ is the unique $H^{++}$-eigenvalue of $\A$,
with the unique positive eigenvector, up to a positive scalar.

\item{\em(Chang, Pearson and Zhang \cite{CPZ})} If moreover $\A$ is irreducible, then $\rho(\A)$ is the unique $H^{+}$-eigenvalue of $\A$,
with the unique nonnegative eigenvector, up to a positive scalar.
\end{enumerate}
\end{thm}

\begin{lem}\cite{YY3} \label{ev}
Let $\A$ be a weakly irreducible nonnegative tensor.
Let $y$ be an eigenvector of $\A$ corresponding to an eigenvalue $\la$ with $|\la|=\rho(\A)$.
Then $|y|$ is the unique positive eigenvector corresponding to $\rho(\A)$ up to a scalar.
\end{lem}

According to the definition of tensor product in \cite{Shao}, for a tensor $\A$ of order $m$ and dimension $n$, and two diagonal matrices $P,Q$ both of dimension $n$,
the product $P\A Q$ has the same order and dimension as $\A$, whose entries are given by
\[(P\A Q)_{i_1i_2\ldots i_m}=p_{i_1i_1}a_{i_1i_2\ldots i_m}q_{i_2i_2}\ldots q_{i_mi_m}.\]
If $P=Q^{-1}$, then $\A$ and $P^{m-1}\A Q$ are called \emph{diagonal similar}.
It is proved that two diagonal similar tensors have the same spectrum \cite{Shao}.

\begin{thm}\label{PF2}\cite{YY3}
Let $\A$ and $\B$ be $m$-th order $n$-dimensional tensors with $|\B|\le \A$. Then
\begin{enumerate}

\item[(1)] $\rho(\B)\le \rho(\A)$.

\item[(2)] If $\A$ is weakly irreducible and $\rho(\B)=\rho(\A)$, where $\la=\rho(\A)e^{\i\theta}$ is
an eigenvalue of $\B$ corresponding to an eigenvector $\y$, then $\y=(y_1, \cdots, y_n)$ contains no zero entries, and $\A=e^{-\i\theta}D^{-(m-1)}\B D$,
where $D=\diag(\frac{y_1}{|y_1|}, \cdots, \frac{y_n}{|y_n|})$.
\end{enumerate}
\end{thm}

Lemma \ref{ev} and Theorem \ref{PF2} were given by Yang and Yang \cite{YY3} and posted in arXiv, and also be rewritten in the Appendix of \cite{FBH} for the completeness of the paper.




\begin{Def}\cite{FHB}\label{ell-sym}
Let $\A$ be an $m$-th order $n$-dimensional tensor, and let $\ell$ be a positive integer.
The tensor $\A$ is called \emph{spectral $\ell$-symmetric} if
\begin{equation}\label{sym-For}\Spec(\A)=e^{\i \frac{2\pi}{\ell}}\Spec(\A).\end{equation}
The maximum number $\ell$ such that (\ref{sym-For}) holds is called the \emph{cyclic index} of $\A$ \cite{CPZ2}.
\end{Def}


%

We now introduce a group associated with a general tensor $\A$ of order $m$ and dimension $n$ \cite{FHB}.
Let $\ell$ be a positive integer.
For $j=0,1,\ldots,\ell-1$, define
\begin{equation} \label{Dg}
\begin{split}
\Dg^{(j)}&=\{D: \A=e^{-\i \frac{2\pi j}{\ell}}D^{-(m-1)}\A D, d_{11}=1\},\\
\Dg &=\cup_{j=0}^{\ell-1}\Dg^{(j)},
\end{split}
\end{equation}
where  $D$ is an $n \times n$ invertible diagonal matrix in above definition.
Sometimes we use $\Dg(\A)$ and $\Dg^{(j)}(\A)$ to avoid confusions.

\begin{lem}
Let $\A$ be an $m$-th order $n$-dimensional nonnegative weakly irreducible tensor.
Let $\Dg^{(j)}$ be as defined in (\ref{Dg}) for $j=1,\ldots,\ell-1$.
Then the following conditions are equivalent.

\begin{enumerate}

\item $\A$ is spectral $\ell$-symmetric.

\item $\la_j=\rho(\A)e^{\i \frac{2\pi j}{\ell}}$ is an eigenvalue of $\A$ for $j=1,\ldots,\ell-1$.

\item $\Dg^{(j)} \ne \emptyset$ for $j=1,\ldots,\ell-1$.

\end{enumerate}

\begin{proof}
If the condition (1) holds, then $\rho(\A)e^{\i \frac{2\pi}{\ell}}$ is an eigenvalue of $\A$.
So, by Theorem \ref{PF2}(2), there exists a diagonal matrix $D$ such that $\A=e^{-\i\frac{2\pi}{\ell}}D^{-(m-1)}\A D$,
implying that $\A=e^{-\i\frac{2\pi j}{\ell}}D^{-j(m-1)}\A D^j$ for $j=1,\ldots,\ell-1$.
Hence the conditions (2) and (3) both hold.
If (2) holds, then (3) holds by Theorem \ref{PF2}(2), which implies that (1) holds by considering $\Dg^{(1)}$.
\end{proof}

\end{lem}


By Lemma \ref{group}, $\Dg$ is an abelian group containing $\Dg^{(0)}$ as a subgroup.
Furthermore, the groups  $\Dg$ and $\Dg^{(0)}$ are only determined by the support of $\A$ \cite{FBH}.
So, they are both combinatorial invariants of $\A$.

\begin{Def}\label{stab}\cite{FBH}
Let $\A$ be an $m$-th order $n$-dimensional tensor.
The \emph{stabilizing index} of $\A$, denoted by $s(\A)$, is defined as the cardinality of the group $\Dg^{(0)}(\A)$.
\end{Def}

Suppose that $\A$ is a nonnegative weakly irreducible tensor of dimension $n$, which is spectral $\ell$-symmetric.
Let $\la_j=\rho(\A)e^{\i \frac{2 \pi j}{\ell}}$ be the eigenvalue of $\A$, $j=0,1,\ldots,\ell-1$.
By Lemma \ref{ev} for any $\x \in \PV_{\la_j}$, $|\x|$ is the unique positive Perron vector of $\A$ up to a scalar.
Define
\begin{equation}\label{Dy}
D_\x=\diag\left(\frac{x_1}{|x_1|},\cdots,\frac{x_n}{|x_n|}\right).
\end{equation}
The following lemma will show that $\Dg^{(j)}(\A)$ is closely related to $\PV_{\la_j}$, and the stabilizing index $s(\A)=|\PV_{\rho(\A)}|$.


\begin{lem}\cite{FHB,FBH} \label{group}
Let $\A$ be an $m$-th order $n$-dimensional nonnegative weakly irreducible tensor, which is spectral $\ell$-symmetric.
Let $\la_j=\rho(\A)e^{\i \frac{2 \pi j}{\ell}}$ be the eigenvalue of $\A$, $j=0,1,\ldots,\ell-1$.
Let $\Dg$ and $\Dg^{(j)}$ be as defined in (\ref{Dg}) for $j=0,1,\ldots,\ell-1$.
Then the following results hold.

\begin{enumerate}

\item $\Dg$ is an abelian group under the usual matrix multiplication, where $\Dg^{(0)}$ is a subgroup of $\Dg$, and
$\Dg^{(j)}$ is a coset of $\Dg^{(0)}$ in $\Dg$ for $j \in [\ell-1]$.

\item there is a bijection between $\PV_{\la_j}(\A)$ and $\Dg^{(j)}(\A)$, and
\begin{equation} \label{2nd}
\Dg^{(j)}(\A)=\{D_\x: \x \in \PV_{\la_j}(\A), x_1=1\}, j=0,1,\ldots,\ell-1.
\end{equation}

\item If further $\A$ is combinatorial symmetric, then $\ell\mid m$, and $D^{m}=\I$ for any $D \in \Dg$.
\end{enumerate}
\end{lem}

Now we arrive at the main result of this paper.
\begin{thm}\label{main}
Let $\A$ be an $m$-th order $n$-dimensional nonnegative weakly irreducible tensor.
Then $\PV_{\rho(\A)}$ has dimension zero, i.e. there are finite many eigenvectors of $\A$ corresponding to $\rho(\A)$ up to a scalar.
\end{thm}

\begin{proof}
Note that $\PV_{\rho(\A)}$ is a projective variety over the field $k=\mathbb{C}$;
 and for any $\x \in \PV_{\rho(\A)}$, $\x$ contains no zero entries by Lemma \ref{ev}.
Here we denote $\x=[x_1:\cdots:x_n]$ for the class of $\x$.
Similar to the map $\phi_0$ defined in (\ref{A0A}), we define
$$ f: \mathbb{P}^{n-1} \to \mathbb{A}^{n-1}, [x_1:x_2:\cdots:x_n] \mapsto (x_2/x_1,\ldots,x_n/x_1).$$
Then $f$ is a regular map; indeed $f$ is an isomorphism.
Restricting $f$ to $\PV_{\rho(\A)}$, $f(\PV_{\rho(\A)})$ is a finite set by Corollary \ref{finite}.
As $f$ is also a bijection, we get $\PV_{\rho(\A)}$ is finite.
\end{proof}

\begin{cor}\label{main2}
Let $\A$ be an $m$-th order $n$-dimensional nonnegative weakly irreducible tensor, which is spectral $\ell$-symmetric.
Then for each eigenvalue $\la_j=\rho(\A)e^{\i \frac{2\pi j}{\ell}}$  for $j \in [\ell-1]$, $|\PV_{\la_j}(\A)|=|\PV_{\rho(\A)}|$,
in particular there are finite many eigenvectors of $\A$ corresponding to $\la_j$ up to a scalar.
\end{cor}

\begin{proof}
By the Lemma \ref{group}(2), for $j=0,1,\ldots, \ell-1$, $|\PV_{\la_j}(\A)|=|\Dg^{(j)}(\A)|$;
and by Lemma \ref{group}(1), $\Dg^{(j)}(\A)$ is a coset of $\Dg^{(0)}(\A)$, implying $|\Dg^{(j)}(\A)|=|\Dg^{(0)}(\A)|$.
So, by Theorem \ref{main}, we get $|\PV_{\la_j}(\A)|=|\PV_{\rho(\A)}|$ for $j\in [\ell-1]$.
\end{proof}

\section{Upper bounds of stabilizing index}
We have proved that for a nonnegative weakly irreducible tensor $\A$,
the stabilizing index $s(\A)=|\PV_{\rho(\A)}|$, which is finite and also means $\PV_{\rho(\A)}$ has dimension zero.
We now discuss the upper bound of $s(\A)$ for some special classes of nonnegative tensors $\A$ when $\A$ is at least weakly irreducible.

\begin{Def}\cite{CPZ, Pea2} A nonnegative matrix $M(\A)$ is called the \emph{majorization}
associated to the nonnegative tensor $\A=(a_{i_1}\ldots a_{i_m})$ of order $m$ and dimension $n$, if
the $(i, j)$-th entry of $M(\A)$ is defined to be $a_{ij\ldots j}$ for $i,j \in [n]$.
\end{Def}

\begin{Def}\cite{LiYY}
A tensor $I(\A)=(b_{i_1\ldots i_m})$ is called \emph{induced to} the tensor $\A=(a_{i_1}\ldots a_{i_m})$ of order $m$ and dimension $n$, if
$b_{i_1\ldots i_m}=a_{i_1}\ldots a_{i_m}$ if $i_2=\cdots=i_m=i$ for $i \in [n]$, and $b_{i_1\ldots i_m}=0$ otherwise.
\end{Def}

\begin{Def} Let $\A=(a_{i_1}\ldots a_{i_m})$ be a nonnegative tensor of order $m$ and dimension $n$. Then

\begin{enumerate}
\item \cite{Pea} $\A$ is called \emph{essentially positive}, if for any $i,j \in [n]$, $a_{ij\ldots j}>0$ or equivalently $\A x^{m-1}>0$ for any $x \gneq 0$.

\item \cite{ZQX} $\A$ is called \emph{weakly positive}, if for any $i,j \in [n]$ and $i \ne j$, $a_{ij\ldots j}>0$.

\item \cite{CPZ3,Pea2} $\A$ is called \emph{primitive}, if for some positive integer $r$ such that $T_\A^r(x)>0$ for any $x \gneq 0$, where $T_\A^r:=T_\A(T_\A^{r-1})$ and
$T_\A(x):=(\A x^{m-1})^{[\frac{1}{m-1}]}$.

\item \cite{FGH} $\A$ is called \emph{weakly primitive}, if the directed graph $G(\A)$ associated with $\A$ is strongly connected and the g.c.d of the lengths of its circuits is equal to one.

\item \cite{LiYY} $\A$ is called \emph{strongly irreducible} if $I(\A)$ is irreducible.

\item \cite{LiYY} $\A$ is called \emph{strongly primitive} if $I(\A)$ is primitive.

\end{enumerate}

\end{Def}

For a nonnegative tensor $\A$, it was shown in \cite[Lemma 1]{LiYY} that $I(\A)$ is irreducible (respectively, primitive) if and only if $M(\A)$ is irreducible
(respectively, primitive).
So, strongly irreducibility (respectively, strongly primitivity) implies irreducibility (respectively, primitivity) as $\A \ge I(\A)$.
Also, irreducibility implies weakly irreducibility by their definitions, and primitivity implies weakly primitivity by \cite[Lemma 3.3]{HHQ}

By the definitions, if $\A$ is essential positive (respectively, weakly positive), then $\A$ is strongly primitive (respectively, strongly irreducible).
A example is given that a tensor $\A=(a_{ijk})$ of order $3$ and dimension $2$ is weakly positive by not primitive as follows (\cite[Example 2.5]{Pea2}, \cite[Example 3.4]{ZQX}):
$$a_{122}=a_{211}=1, a_{ijk}=0 \hbox{~elsewhere}.$$
But a weakly positive tensor of dimension greater than $2$ is strongly primitive.

It is known that a nonnegative matrix $A$ is irreducible if and only if $A+\I$ is primitive.
So, $\A$ is strongly irreducible (respectively, irreducible, weakly irreducible) if and only if $\A+\I$ is strongly primitive by \cite[Lemma 1]{LiYY}
(respectively, primitive by \cite[Corollary 3.8]{CPZ3}, \cite[Theorem 6.1]{YY2} and \cite[Theorem 2.4]{Pea2}, weakly primitive by definition).

It is easily seen that $\A=D^{-(m-1)}\A D$ if and only if  $\A+\I=D^{-(m-1)}(\A+\I) D$ for a tensor $\A$ of order $m$.
So $\Dg^{(0)}(\A)=\Dg^{(0)}(\A+\I)$, and $s(\A)=s(\A+\I)$.
From this fact the irreducibility is enough to discuss the stabilizing index of nonnegative tensors, as the primitivity seems give us no more benefit.



\begin{thm} \cite{FBH} \label{wsym}
Let $\A$ be a nonnegative combinatorially symmetric weakly irreducible tensor of order $m$ and $n$ dimension.
Then $s(\A)|m^{n-1}$, $s(\A)< m^{n-1}$.
%
\end{thm}

\begin{thm}\cite{LiYY} \label{LYY}
Let $\A$ be a nonnegative irreducible tensor of order $m$ and dimension $n$, let
$G$ be the set of eigenvectors corresponding to $\rho(\A)$. If $\x \in G$, then
$$\x^{[m-1]^r}=e^{\i\theta}|\x|^{[m-1]^r},$$
where $\theta$ relies on $x$, and $r$ is the number of irreducible blocks of the majorization matrix $M(\A)$.
In particular, if $\A$ is also strongly irreducible, then
$$\x^{[m-1]}=e^{\i\theta}|\x|^{[m-1]}.$$
\end{thm}

If using our language, Theorem \ref{LYY} says for any $D \in  \Dg^{(0)}(\A)$, $D^{[m-1]^r}=\I$ by normalizing $x_1=1$.
Next we will give a more refining result on the power of $\x$ in  Theorem \ref{LYY} by using graph language.

Let $\A$ be a nonnegative irreducible tensor of order $m$ and dimension $n$.
Considering the directed graph $G(\A)$ defined in Section 3.1, which is strongly connected as $\A$ is also weakly irreducible.
If an arc $(i,j)$ of $G(\A)$ arise from the nonzero entry $a_{ij\ldots j}$, it is called a \emph{solid arc}; otherwise,  it is called a \emph{dotted arc}.
The \emph{solid graph} associated with $\A$ is defined to be subgraph of $G(\A)$ induced by the solid arcs, denoted by $G^s(\A)$.
As $\A$ is irreducible, for every $j \in [n]$, there is a solid arc point to $j$, i.e. an arc of form $(i,j)$ for some $i \ne j$.
So, $G^s(\A)$ has vertex set $[n]$.
If we define a directed graph $G(M(\A))$ based on the majorization matrix $M(\A)$ similarly, that is, $(i,j)$ is an arc of $G(M(\A))$ if and only if $M(\A)_{ij} \ne 0$ or equivalently
$a_{ij\ldots j}\ne 0$, then $G(M(\A))=G^s(\A)$.
The number of irreducible blocks of $M(\A)$ is exactly the number of the strongly connected components of $G^s(\A)$.

If we think of $G^s(\A)$ as an undirected graph by ignoring the direction of arcs, then a connected component of such undirected graph is called a \emph{weakly connected component} of $G^s(\A)$.
If $G^s(\A)$ contains only one weakly connected component, then it is called \emph{weakly connected}.
Obviously, the number of weakly connected components of $G^s(\A)$ is not larger than that of the strongly connected components of $G^s(\A)$ or that of the irreducible blocks of $M(\A)$.

Suppose $G^s(\A)$ has $k$ weakly connected components, say $C_1,\ldots,C_k\;(k \ge 1)$.
Returning to the original graph $G(\A)$, there exist only dotted arcs between any $C_i$ and $C_j$ for $i \ne j$.


\begin{thm}\label{gener}
Let $\A$ be a nonnegative irreducible tensor of order $m$ and dimension $n$,  let
$\x$ be an eigenvector of $\A$ corresponding to $\rho(\A)$ and $D_\x$ be defined as in (\ref{Dy}).
Then
$$\x^{[m-1]^{r}}=e^{\i\theta}|\x|^{[m-1]^{r}}, D_\x^{(m-1)^{r}}=e^{\i \theta}\I$$
where $\theta$ relies on $x$, and $r$ is the number of weakly connected components of the solid graph $G^s(\A)$.
In particular, if $G^s(\A)$ is weakly connected, then
$$\x^{[m-1]}=e^{\i\theta}|\x|^{[m-1]}, D_\x^{(m-1)}=e^{\i \theta}\I.$$
\end{thm}

\begin{proof}
Let $D=D_\x:=\diag(e^{\i \theta_1}, \ldots, e^{\i \theta_n})$.
By Theorem \ref{PF2}(2), $\A=D^{-(m-1)}\A D$, or equivalently
for any $i_1,\ldots,i_m \in [n]$,
$$ a_{i_1 \ldots i_m}=e^{\i (-(m-1)\theta_{i_1})} a_{i_1 \ldots i_m} e^{\i \theta_{i_2}} \cdots e^{\i \theta_{i_m}}.$$
Then, if $ a_{i_1 \ldots i_m} \ne 0$, we have
\begin{equation}\label{similar}
(m-1)\theta_{i_1} \equiv ( \theta_{i_2}+\cdots+\theta_{i_m}) \mod 2 \pi.
\end{equation}

First suppose that the solid graph $G^s(\A)$ has only one weakly connected component, i.e. it is weakly connected.
By (\ref{similar}), for an arc $(i,j)$ or $(j,i)$ of $G^s(\A)$, we get
$$(m-1) \theta_i \equiv (m-1) \theta_j \mod 2 \pi.$$
As $G^s(\A)$ is weakly connected, for any two vertices there exists a sequences of arcs connecting them by ignoring the direction of the arcs.
So, we have
$$\theta:=(m-1) \theta_1  \equiv \cdots \equiv (m-1)\theta_n \mod 2\pi,$$
which yields $D^{m-1}=e^{\i \theta}\I$, and hence $\x^{[m-1]}=e^{\i\theta}|\x|^{[m-1]}$.

Next suppose that $G^s(\A)$ has $k\;(k \ge 2)$ weakly connected components.
We label the  weakly connected components of $G^s(\A)$ as $C_1,\ldots,C_k$ under the following procedure:
\begin{enumerate}
\item  $C_1$ is any weakly connected component.
\item Suppose that $C_i$ is labelled  for $i \ge 1$ and $\cup_{j\le i} C_j \subsetneq [n]$ , there exists a vertex $p$ outside $\cup_{j\le i} C_j$ such that there are $m-1$ dotted arcs arising from a nonzero
  entry $a_{pi_2\ldots i_m}\ne 0$, where $i_2,\ldots,i_m \in \cup_{j \le i} C_j$.
Then $C_{i+1}$ is any weakly connected component containing $p$.
\end{enumerate}
The labelling process is available as $\A$ is irreducible and if taking $I^c=\cup_{j\le i} C_j$
there must exists an element $p \in I$ such that $a_{pi_2\ldots i_m}\ne 0$, where $i_2,\ldots,i_m \in I^c$.

Let $D_{C_i}$ be principal submatrix of $D$ indexed by the vertices of $C_i$ for $i=1,\ldots,k$.
By what have proved,
\begin{equation} \label{similar2} D_{C_i}^{m-1}=e^{\i \omega_i}\I \hbox{~for some~} \omega_i \hbox{~and for~} i=1,\ldots,k.\end{equation}

If $r=2$, then there exist $m-1$ dotted arcs of $G(\A)$ arising from
a nonzero entry $a_{i_1 i_2\ldots i_m}$, where $i_1$ is a vertex of $C_2$, and $i_2,\ldots,i_m$ are vertices of $C_1$ which are not all same.
So, by (\ref{similar}) and (\ref{similar2}), we have
$$ \theta:=(m-1) \omega_2 \equiv (m-1)^2\theta_{i_1} \equiv (m-1)\theta_{i_2}+\cdots+(m-1)\theta_{i_m} \equiv (m-1)\omega_1 \mod 2 \pi,$$
which implies that $D^{(m-1)^2}=e^{\i \theta}\I$.

Assume that the result holds for $k=t\;(t \ge 1)$.
When $k=t+1$, let $G_1$ be the subgraph of $G(\A)$ induced by the vertices of $C_1,\ldots, C_t$.
By the assumption, we have
\begin{equation} \label{similar3} D_{G_1}^{(m-1)^t}=e^{\i \tau} \I \hbox{~for some ~} \tau.\end{equation}
There exist $m-1$ dotted arcs of $G(\A)$ arising from
a nonzero entry $a_{i_1 i_2\ldots i_m}$, where $i_1 \in C_{t+1}$, and $i_2,\ldots,i_m \in \cup_{j \le i}C_j$ which are not all same.
Similar to the above discussion, we have
\begin{eqnarray*}
\theta &:=& (m-1)^t \omega_{t+1}\equiv (m-1)^{t+1} \theta_{i_1} \\
& \equiv & (m-1)^{t}\theta_{i_2} + \cdots + (m-1)^{t} \theta_{i_m}\\
& \equiv & (m-1) \tau \mod 2 \pi,
\end{eqnarray*}
which implies that $D^{(m-1)^{t+1}}=e^{\i \theta}\I$.
\end{proof}

\begin{cor}
Let $\A$ be a nonnegative irreducible tensor of order $m$ and dimension $n$.
Then for any $D \in \Dg^{(0)}(\A)$, $D^{(m-1)^r}=I$, where $r$ is the number of weakly connected components of $G^s(\A)$; and
$s(\A) \le (m-1)^{r(n-1)}$.
\end{cor}

\begin{proof}
By Lemma \ref{group}(2), for any $D\in \Dg^{(0)}(\A)$, $D=D_\x$ for some $\x \in \PV_{\rho(\A)}$ with $\x_1=1$.
So, in Theorem \ref{gener}, if normalizing $\x$ so that $x_1=1$, then $\D_\x^{(m-1)^r}=\I$.
As $D^{(m-1)^r}=\I$, we can write
$$D=\hbox{diag}\left(1,e^{\i \frac{2 \pi t_1}{(m-1)^r}}, \ldots, e^{\i \frac{2 \pi t_{n-1}}{(m-1)^r}}\right)$$
for some integers $t_1,\ldots,t_{n-1} \in [(m-1)^r]$.
The number of choices of such $D$ is at most $(m-1)^{r(n-1)}$. The result follows.
\end{proof}

\begin{cor}
Let $\A$ be a nonnegative combinatorially symmetric irreducible tensor of order $m$ and dimension $n$.
Then $\A$ has only one eigenvector corresponding to $\rho(\A)$, i.e. the Perron vector up to scalar, or equivalently $s(\A)=1$.
\end{cor}

\begin{proof}
By Lemma \ref{group}(2), $|\PV_{\rho(\A)}|=|\Dg^{(0)}(\A)|$.
For any $D \in \Dg^{(0)}(\A)$, $D^m=\I$ by Lemma \ref{group}(3),
and $D^{(m-1)^r}=\I$ for some positive integer $r$ by Theorem \ref{gener} .
As $\hbox{g.c.d}(m,m-1)=1$, $D=\I$. The result follows.
\end{proof}

\section{Projective eigenvariety of combinatorially symmetric tensors}

Let $\A$ be a nonnegative tensor of order $m$ and dimension $n$.
Consider the directed graph $G=G(\A)$ associated with $\A$.
For $i,j \in [n]$, we say that \emph{$i$ has an access to $j$} if there is a directed path from $i$ to $j$,
and that $i$ and $j$ \emph{communicate} if $i$ has an access to $j$ and $j$ has an access to $i$.
So we have equivalent classes on $[n]$ of the communication relation induced by $G$.
Each class $\alpha$ corresponds to a strongly connected component $G[\alpha]$ induced by $\alpha$, and vice versa.
A class $\alpha$ \emph{has an access to a class $\beta$} if for $i\in \alpha$ and $j \in \beta$, $i$ has an access to $j$.
A class is called \emph{final} if it has no access to any other class.

For the graph $G$, we can associated with a square matrix $A(G)=(a_{ij})$ indexed by $[n]$ such that $a_{ij}=1$ if $(i,j)$ is an arc of $G$, and $a_{ij}=0$ otherwise.
Suppose $[n]$ has $s$ classes, say $\alpha_1,\ldots,\alpha_s$, which are labelled such that
for each $i=1,\ldots,s$, $\alpha_i$ is the final class in the subgraph $G[\cup_{j \le i} \alpha_j]$.
So, by a suitable permutation $P$,  $P^TA(G)P$ has the following form:
\begin{equation}\label{matrixform}
P^TA(G)P=\left[
\begin{array}{cccc}
A_{11} & A_{12} & \cdots & A_{1s}\\
0 & A_{22} & \cdots & A_{2s}\\
\vdots  & \vdots & \ddots & \vdots\\
0 &0 & \cdots & A_{ss}
\end{array}
\right],
\end{equation}
where $A_{ij}=:A(G)[\alpha_i|\alpha_j]$, the submatrix of $A(G)$ with rows indexed by $\alpha_i$ and columns indexed by $\alpha_j$;
in particular $A[\alpha_i]:=A[\alpha_i|\alpha_i]=A(G[\alpha_i])$, which is an irreducible block corresponding to the strongly connected components $G[\alpha_i]$.

Now returning to the tensor $\A$.
We also get the subtensors $\A[\alpha_{i_1}|\alpha_{i_2}|\cdots|\alpha_{i_m}]$, whose entries are $a_{i_1i_2\cdots i_m}$ with $i_j \in \alpha_j$ for $j \in [m]$.
For brevity, if $\alpha_{i_j}=:\alpha$ for each $j \in [m]$, then we denote the above subtensor as $\A[\alpha]$, called the \emph{principal subtensor} of $\A$.
By (\ref{matrixform}), we get the structure of $\A$, that is,
\begin{equation}\label{tensorform}
\A[\alpha_{i_1}|\alpha_{i_2}|\cdots|\alpha_{i_m}] \ne 0 \hbox{~only when~} i_1 \le i_j \hbox{~for~} j=2,\ldots,m \hbox{~and for~} i_1 \in [s].
\end{equation}
Note that $\A[\alpha_i]$ may not be weakly irreducible.
For example, let $\A=(a_{ijk})$ be a tensor of order $3$ and dimension $4$ whose nonzero entries are given as follows:
$$ a_{111}=a_{124}=a_{234}=a_{341}=1.$$
Then we have two classes: $\alpha_1=\{1,2,3\}$ and $\alpha_2=\{4\}$.
It is easy to see that $\A[\alpha_1]$ and $\A[\alpha_2]$ are both weakly reducible.

\begin{lem}\label{reddsp}
Let $\A$ be a nonnegative tensor with structure (\ref{tensorform}).
Then $\rho(\A[\alpha_i]) \le \rho(\A)$ for $i=1,\ldots,s$.
If $\x$ is an eigenvector of $\A$ corresponding to $\rho(\A)$, then there exists at least one $\alpha_i$ such that $\rho(\A[\alpha_i]) = \rho(\A)$ and $\x[\alpha_i] \ne 0$.
\end{lem}

\begin{proof}
By Theorem \ref{PF1}(1), there is a nonnegative eigenvector $\y$ of $\A[\alpha_i]$ corresponding to the spectral radius $\rho(\A[\alpha_i])$,
i.e. $\A[\alpha_i]\y^{m-1}=\rho(\A[\alpha_i]) \y^{[m-1]}$.
Now define a nonnegative vector $\bar{\y} \in \mathbb{R}^n$ such that $\bar{\y}[\alpha_i]=\y$ and $\bar{\y}[\alpha_i^c]=0$.
By \cite[Theorem 5.3]{YY1},
$$\rho(\A)=\max_{\x \gneq 0, \x \in \mathbb{R}^n} \min_{x_i>0} \frac{(\A \x^{m-1})_i}{x_i^{m-1}}
 \ge \min_{\bar{y}_i>0} \frac{(\A \bar{\y}^{m-1})_i}{\bar{y}_i^{m-1}}=
\min_{y_i>0} \frac{(\A[\alpha_i] \y^{m-1})_i}{y_i^{m-1}}=\rho(\A[\alpha_i]).$$

Now suppose $\x$ is an eigenvector of $\A$ corresponding to $\rho(\A)$.
If $\x[\alpha_s]\ne 0$, then by the eigenvector equation $\A\x^{m-1}=\rho(\A) \x^{[m-1]}$,
 by (\ref{tensorform}), we get
$$  \A[\alpha_s]\x[\alpha_s]^{m-1}=\rho(\A) \x[\alpha_s]^{[m-1]}.$$
So $\rho(\A[\alpha_s])=\rho(\A)$.
Otherwise, $\x[\alpha_s]= 0$.
Let $t$ be the integer such that $x[\alpha_t]\ne 0$ and $x[\alpha_i]=0$ for $t+1 \le i \le s$.
Then $ 1 \le t \le s-1$; and also by the eigenvector equation,
$$ \A[\alpha_t]\x[\alpha_t]^{m-1}+ \sum_{t\le i_2,\ldots,i_m \le s \atop t <\max(i_2,\ldots,i_m)}
\A[\alpha_t|\alpha_{i_2}|\cdots|\alpha_{i_m}]\x[\alpha_{i_2}]\cdots \x[\alpha_{i_m}]=\rho(\A)\x[\alpha_t]^{[m-1]},$$
where, for each $i \in \alpha_t$,
 $$(\A[\alpha_t|\alpha_{i_2}|\cdots|\alpha_{i_m}]\x[\alpha_{i_2}]\cdots \x[\alpha_{i_m}])_i:=\sum_{i_j \in \alpha_j, j=2,\ldots,m}a_{ii_2\ldots, i_m}x_{i_2}\cdots x_{i_m}.$$
So, $\A[\alpha_t]\x[\alpha_t]^{m-1}=\rho(\A)\x[\alpha_t]^{[m-1]}$ and hence $\rho(\A[\alpha_t])=\rho(\A)$.
\end{proof}

Suppose that $\A$ is further combinatorially symmetric.
 Then $G(\A)$ is an undirected graph, i.e. $(i,j)$ is an arc of $G$ if and only if
$(j,i)$ is an arc of $G$.
So the matrix $A(G)$ is symmetric, and is a direct sum of irreducible diagonal blocks.
Then the structure of $\A$ is reduced to be as follows:
\begin{equation}\label{tensorform2}
\A[\alpha_{i_1}|\alpha_{i_2}|\cdots|\alpha_{i_m}] \ne 0 \hbox{~only when~} i_1=\cdots =i_m.
\end{equation}
So $\A$ is a direct sum of combinatorially symmetric tensors $A[\alpha_i]$ for $i=1,\ldots,s$.
Each $\A[\alpha_i]$ is either weakly irreducible or zero of dimension one,
because if $(p,q)$ is an arc of $G[\alpha_i]$, then $\A$ contains a nonzero entry $a_{pi_2\ldots i_m}$,
where $q \in \{i_2,\ldots, i_m\}$,
and $i_2,\ldots,i_m \in \alpha_i$ as $G[\alpha_i]$ is undirected, yielding to $G[\alpha_i]=G(\A[\alpha_i])$.
If $A[\alpha_i] \ne 0$, we call $\A[\alpha_i]$ a \emph{weakly irreducible component}.

\begin{thm}\label{main3}
Let $\A$ be a nonzero nonnegative combinatorially symmetric tensors.
Then $\PV_{\rho(\A)}$ has dimension greater than zero if and only if
$\A$ has at least two weakly irreducible components with spectral radius equal to $\rho(\A)$.
\end{thm}

\begin{proof}
By the previous discussion, let $\A$ be a direct sum of weakly irreducible components or zeros of dimension one, say $\A[\alpha_i]$ for $i=1,\ldots,s$ for some $s \ge 1$.
By Lemma \ref{reddsp}, we may assume $\A$ has $k$ weakly irreducible components with spectral radius equal to $\rho(\A):=\rho$,
 say $\A[\alpha_1],\ldots,\A[\alpha_k]$ for some $k \ge 1$.
Let $\x$ be an eigenvector of $\A$ corresponding to $\rho$.
Then
$$ \A[\alpha_i]\x[\alpha_i]^{m-1}=\rho \x[\alpha_i]^{[m-1]}, i=1,\ldots,s.$$
So, if $\x[\alpha_i] \ne 0$, then $\rho(\A[\alpha_i])=\rho$ by Lemma \ref{reddsp}, or equivalently
if $\rho(\A[\alpha_i])<\rho$, then $\x[\alpha_i] = 0$.
So we have the isomorphism for the affine variety $\mathcal{V}_{\rho}(\A)$:
$$\mathcal{V}_{\rho}(\A) \cong \mathcal{V}_{\rho}(\A[\alpha_1])\oplus \cdots \oplus \mathcal{V}_{\rho}(\A[\alpha_k]),$$ which has dimension $k$ as each $\mathcal{V}_{\rho}(\A[\alpha_i])$ has dimension $1$ by  Theorem \ref{main}.
Hence the projective variety $\PV_{\rho}(\A)$ has dimension $k-1$.
The result follows.
\end{proof}

\section{Application to hypergraphs}
A {\it hypergraph} $G=(V(G),E(G))$ consists of a vertex set  $V(G)=\{v_1,v_2,\ldots,v_n\}$ and an edge set $E(G)=\{e_{1},e_2,\ldots,e_{l}\}$
  where $e_{j}\subseteq V(G)$ for $j \in [l]$.
If $|e_{j}|=m$ for each $j\in [l]$, then $G$ is called an {\it $m$-uniform} hypergraph.
In particular, the $2$-uniform hypergraphs are exactly the classical simple graphs.
Two vertices $u,w$ of $G$ are said \emph{connected} if there exists a sequence of alternate vertices and edges of $G$: $v_{0},e_{1},v_{1},e_{2},\ldots,e_{l},v_{l}$,
    where $u=v_0$, $w=v_l$, and $\{v_{i},v_{i+1}\}\subseteq e_{i+1}$ for $i=0,1,\ldots,l-1$.
Under the relation of connectedness, we have equivalent classes on $V(G)$.
A \emph{connected component} of $G$ is defined to be a sub-hypergraph of $G$ induced by an equivalent class $\alpha$, denoted by $G[\alpha]$,
which has the vertex set $\alpha$ and the edge set $\{e: e \in E(G), e \subseteq \alpha\}$.
The hypergraph $G$  is {\it connected} if $G$ has only one connected component; otherwise it is called \emph{disconnected}.

The {\it adjacency tensor} $\A(G)$ of an $m$-uniform hypergraph $G$ on $n$ vertices is defined as $\mathcal{A}(G)=(a_{i_{1}i_{2}\ldots i_{m}})$ \cite{CD}, an $m$-th order $n$-dimensional tensor, where
\[a_{i_{1}i_{2}\ldots i_{m}}=\left\{
 \begin{array}{ll}
\frac{1}{(m-1)!}, &  \mbox{if~} \{v_{i_{1}},v_{i_{2}},\ldots,v_{i_{m}}\} \in E(G);\\
  0, & \mbox{otherwise}.
  \end{array}\right.
\]
  The eigenvalues, including the spectral radius of $G$ always refer to those of its adjacency tensor $\A(G)$.

Observe that the adjacency tensor $\A(G)$ is symmetric,
  and it is weakly irreducible if and only if the $G$ is connected \cite{PF,YY3}.
The connected components of $G$ on at least two vertices are exactly corresponding to the weakly irreducible components of $\A(G)$.
So we have the following results for hypergraphs immediately by Theorem \ref{main}, Corollary \ref{main2} and Theorem \ref{main3}.

\begin{thm}
Let $G$ be an $m$-uniform hypergraph containing at least one edge.
Let $\A(G)$ be the adjacency tensor of $G$ with spectral radius $\rho$.

\begin{enumerate}
\item
If $G$ is connected, then for all eigenvalues $\la$ of $G$ with modulus $\rho$, the projective eigenvarieties $\PV_{\la}(\A(G))$ are finite and have the same
number of elements, in particular there are finite many eigenvectors of $\A$ corresponding to $\la$ up to a scalar.

\item If $G$ is disconnected, then $\PV_{\rho}(\A(G))$ has dimension greater than zero if and only if
$G$ has at least two connected components on at least two vertices with spectral radius equal to $\rho$.
\end{enumerate}
\end{thm}

\end{document}